\documentclass[reqno]{amsart}

\usepackage[sc]{mathpazo} 

\usepackage{mathrsfs}
\usepackage{amsmath}
\usepackage{amsthm}
\usepackage{amsfonts}

\newtheorem{thm}{Theorem}[section]
\newtheorem{lem}[thm]{Lemma}
\newtheorem{cor}[thm]{Corollary}
\newtheorem{prop}[thm]{Proposition}
\newtheorem{rems}[thm]{Remarks}
\newtheorem{rem}[thm]{Remark}

\DeclareMathAlphabet{\mathpzc}{OT1}{pzc}{m}{it}

\numberwithin{equation}{section}

\newcommand{\Wqb}{W_{p,N}}
\newcommand{\Wpb}{W_{p,N}}

\newcommand{\R}{\mathbb{R}}

\newcommand{\N}{\mathbb{N}}
\newcommand{\F}{\mathbb{F}}
\newcommand{\A}{\mathbb{A}}

\newcommand{\E}{\mathbb{E}}

\newcommand{\ml}{\mathcal{L}}

\newcommand{\mk}{\mathcal{K}}
\newcommand{\Isom}{\mathrm{Isom}}

\newcommand{\Om}{\Omega}

\newcommand{\ve}{\varepsilon}
\newcommand{\rd}{\mathrm{d}}

\newcommand{\bqn}{\begin{equation}}
\newcommand{\eqn}{\end{equation}}
\newcommand{\bqnn}{\begin{equation*}}
\newcommand{\eqnn}{\end{equation*}}
\newcommand{\bear}{\begin{eqnarray}} 
\newcommand{\eear}{\end{eqnarray}} 
\newcommand{\bean}{\begin{eqnarray*}} 
\newcommand{\eean}{\end{eqnarray*}} 
\newcommand{\bs}{\begin{split}}
\newcommand{\es}{\end{split}}

\newcommand{\dimens}{\mathrm{dim}}
\newcommand{\codim}{\mathrm{codim}}
\newcommand{\kk}{\mathrm{ker}}
\newcommand{\rg}{\mathrm{rg}}

\newcommand{\dhr}{\mathrel{\lhook\joinrel\relbar\kern-.8ex\joinrel\lhook\joinrel\rightarrow}}

\title[Global Continua of Positive Equilibria]
{Global Continua of Positive Equilibria for some Quasilinear Parabolic Equation with a Nonlocal Initial Condition}

\author[Ch. Walker]{Christoph Walker}

\address{Leibniz Universit\"at Hannover, Institut f\"ur Angewandte Mathematik, Welfengarten 1, D--30167 Hannover, Germany.}
\email{walker@ifam.uni-hannover.de}

\begin{document}

\begin{abstract}
This paper is concerned with a quaslinear parabolic equation including a nonlinear nonlocal initial condition. The problem arises as equilibrium equation in population dynamics with nonlinear diffusion. 
We make use of global bifurcation theory to prove existence of an unbounded continuum of positive solutions.
\end{abstract}


\maketitle

\section{Introduction}

Age-structured equations have a long history (see \cite{WebbSpringer} and the references therein), both for the case without and with spatial movement of individuals. In an abstract form, the evolution of an age-structured population with quasilinear diffusion is governed by the equations
\begin{align}
&\partial_tu \, +\partial_au \, +\, A(u,a)\,u\,+\, \mu(u,a)\,u =0\ ,& & t>0\ ,\qquad a\in J\setminus\{0\}\ ,\label{1q}\\ 
&u(t,0)\, =\, \int_0^{a_m} \beta(u,a)\,u(t,a)\,\rd a\ ,& & t>0 \label{2q}
\end{align}
subject to some initial condition at $t=0$. Here, the function $u=u(t,a)\ge 0$ represents the density of a population of individuals with respect to time $t\ge 0$ and age $a\in J$, where $J:=[0,a_m)$ with maximal age $a_m$. For $u$ and $a$ fixed, $A(u,a)$ is a linear unbounded operator $$A(u,a):E_1\subset E_0\rightarrow E_0$$ defined on some common subspace $E_1$ of an ordered Banach space $E_0$. It reflects a nonlinear spatial dispersion of the individuals. Besides spatial movement, the population undergoes death and birth processes with density dependent death modulus $\mu(u,a)\ge 0$ and birth modulus $\beta(u,a)\ge 0$. Specifically, equation \eqref{2q} gives the newborns.

To understand the asymptotic behavior of the evolution of structured populations, precise information about equilibrium (i.e. time-independent) solutions is needed. Clearly, $u\equiv 0$ is an equilibrium solution to \eqref{1q}, \eqref{2q} which therefore has to be singled out in the further analysis. Moreover, since $u$ in \eqref{1q}, \eqref{2q} represents a density, solutions should be nonnegative and thus, in our abstract setting, belong to the positive cone $E_0^+$ of the ordered Banach space $E_0$. 
 
For non-diffusive age-structured populations (i.e. $A\equiv 0$), positive equilibria were established e.g. by means of fixed point theorems in conical shells \cite{Pruess2,Webb}. These results were carried over to the diffusive case as well \cite{WalkerJDE}. However,   the fact that $u\equiv 0$ is a time-independent solution to \eqref{1q}, \eqref{2q} allows one also
 to interpret the problem of finding positive equilibrium solutions as a bifurcation problem \cite{CushingI} and so to obtain more insight on the structure of the equilibria set. For this we write the birth modulus in the form $\beta(u,a)=\lambda b(u,a)$ and thus introduce a bifurcation parameter $\lambda>0$ which determines the intensity of the individual's fertility while the qualitative structure of the fertility  is modeled by the function $b$.
Writing $\A(u)=A(u,\cdot)+\mu(u,\cdot)$ and 
\bqnn\label{agebc}
\lambda\ell(v)u:=\int_0^{a_m} \lambda b(v,a) u(a)\rd a\ ,
\eqnn
we are concerned in this paper with finding values $\lambda>0$ and positive nontrivial solutions $u:J\rightarrow E_0^+$ to the nonlinear problem
\begin{align}
&\partial_au \, +\, \A(u)\,u\, =0\ ,\qquad a\in J\setminus\{0\}\ ,\label{1}\\ 
&u(0)\, =\, \lambda\,\ell (u)u\ .\label{2}
\end{align}
The results derived herein extend our previous results \cite{WalkerSIMA,WalkerJDE} on local and global bifurcation. 
In \cite{WalkerSIMA} it was shown that if $\A(u)$ depends sufficiently smooth on $u$ and if for $u=0$, the operator $\A(0)$ possesses maximal $L_p$-regularity (for a precise definition see the next section), then a local curve of positive solutions to \eqref{1}, \eqref{2} bifurcates from the trivial branch $(\lambda,u)=(\lambda,0)$, \mbox{$\lambda\in\R$}. This local branch was subsequently extended  in \cite{WalkerJDE} to a global continuum by applying Rabinowitz' global alternative \cite{Rabinowitz}, but for less general diffusion operators. More precisely, the existence of an {unbounded} continuum of nontrivial positive solutions $(\lambda,u)$ to \eqref{1}, \eqref{2} was derived under the assumption that the operator $\A(u)$ admits a suitable decomposition $\A(u)=\A_0+\A_*(u)$ with $\A_*$ being of ``lower order''. Although the operator $\A(u)$ may still depend nonlinearly on $u$, a quasilinear dependence, however, is not covered by the bifurcation result of \cite{WalkerJDE} and hence, if $\A(u)$ is a second order differential operator, nonlinearities are merely allowed in the zero and first order terms.

The aim of this paper is to remedy this deficiency by establishing a global continuum of positive solutions for quasilinear diffusion operators $\A(u)$. After recalling (and refining) in Section~\ref{Sec2} the results of \cite{WalkerSIMA} on local bifurcation, we shall show in Section~\ref{Sec4} global bifurcation from the trivial branch provided some convexity condition (see \eqref{A2convex} for details) holds implying maximal $L_p$-regularity of $\A(u)$ for each $u$. The proof relies on a recent result of Shi $\&$ Wang~\cite{ShiWang} which is based on the unilateral global bifurcation techniques of Rabinowitz~\cite{Rabinowitz} or rather their interpretation by L\'opez-G\'omez~\cite{LopezGomez}.
 As we shall see then in  Section~\ref{Sec3}, the convexity condition \eqref{A2convex} is not necessary provided $\A$ and $\ell$ in \eqref{1}, \eqref{2}  depend real analytically on $u$. Indeed, in this case, the analytic bifurcation theory due to Buffoni $\&$ Toland~\cite{BuffoniToland} yields   a global smooth curve of positive solutions to \eqref{1}, \eqref{2}. Finally, in Section~\ref{Sec5} we give an example of a concrete operator and the corresponding function space setting. Further examples to which the results of the present paper apply can be found in \cite{WalkerSIMA,WalkerJDE}. Applications of global bifurcation theory to nonlinear population models are given e.g. in~\cite{WalkerCrelle}.

\section{Preliminaries}\label{Sec2}


\subsection{General assumptions.} 
If $E$ and $F$ are Banach spaces we write $\ml(E,F)$ for the set of linear bounded operators from $E$ to $F$, and we put $\ml(E):=\ml(E,E)$. The subset thereof consisting of compact operators is denoted by $\mk(E,F)$ and $\mk(E)$, respectively. 
$\Isom(E,F)$ stands for the set of topological isomorphisms~$E\rightarrow F$. By $E\dhr F$ we mean that $E$ is compactly embedded in $F$.\\

Throughout the paper we assume that
$E_0$ is a real Banach space ordered by a closed convex cone $E_0^+$ and $E_1$ is an embedded Banach space such that 
\bqn\label{densecompact}
\text{the embedding } E_1\hookrightarrow E_0\ \text{is dense and compact}\ .
\eqn
We fix $p\in (1,\infty)$ and set $E_\varsigma:=(E_0,E_1)_{\varsigma,p}$ with $(\cdot,\cdot)_{\varsigma,p}$ being the real interpolation functor for $\varsigma:=\varsigma(p):=1-1/p$.
For each $\theta\in (0,1)\setminus\{1-1/p\}$ we let $(\cdot,\cdot)_\theta$ denote  an admissible interpolation functor, that is, an interpolation functor $(\cdot,\cdot)_\theta$ such that the embedding $$E_1 \hookrightarrow E_\theta:=(E_0,E_1)_\theta$$ is dense. Note that the embedding $E_\theta\dhr E_\vartheta$ is compact for $0\le\vartheta<\theta\le 1$ (see \cite[I.Thm.2.11.1]{LQPP}). The interpolation spaces $E_\theta$, $0\le \theta\le 1$ are given their natural order induced by the cone $E_\theta^+:=E_\theta\cap E_0^+$. We suppose that 
\bqn\label{positiv}
\mathrm{int}(E_\varsigma^+)\not= \emptyset\ .
\eqn
Recall that $a_m\in (0,\infty]$ and set $J:=[0,a_m)$. Observe that $a_m=\infty$ is explicitly allowed. We introduce the spaces
$$
\E_0:=L_p(J,E_0)\ ,\qquad \E_1:=L_p(J,E_1)\cap W_p^1(J,E_0)
$$
and recall the embedding
\bqn\label{emb}
\E_1\hookrightarrow BUC(J,E_\varsigma)\ ,
\eqn
where $BUC$ stands for the bounded and uniformly continuous functions. Thus, the trace $$\gamma_0 u:=u(0)\ ,\quad u\in \E_1\ ,$$ yields a well-defined operator $\gamma_0\in\ml(\E_1,E_\varsigma)$. We let $\E_1^+:=L_p^+(J,E_1)\cap W_p^1(J,E_0)$ denote the positive cone of $\E_1$ and put $\dot{\E}_1^+:=\E_1^+\setminus\{0\}$.
We fix a Banach space  $\F$ such that 
\bqn\label{comp}
\E_1\dhr\F
\eqn 
and let $\Sigma$ denote an open connected zero-neighborhood in $\F$.  Then $\Sigma_1:=\Sigma\cap \E_1$ is an open connected zero-neighborhood  in $\E_1$. Suppose that for some $\vartheta\in(\varsigma,1]$ we have\footnote{Observe that this notation includes that $A=A(u,a)$ in \eqref{1q} depends in a local way on age $a$.}
\bqn\label{A1}
\A\in C^1\big(\Sigma, \ml(\E_1,\E_0)\big)\ ,\qquad \ell \in C^1\big(\Sigma, \ml(\E_1,E_\vartheta)\big)
\eqn
and that for each $u\in\Sigma_1$, the operator $\A(u)$ possesses maximal $L_p$-regularity, that is,
\bqn\label{A2}
\big(\partial_a+\A(u),\gamma_0\big)\in\Isom(\E_1,\E_0\times E_\varsigma\big)\ ,\quad u\in \Sigma_1\ .
\eqn 
Then
\bqn\label{T}
\big(u\mapsto T[u]:=\big(\partial_a+\A(u),\gamma_0\big)^{-1}\big) \in C(\Sigma_1,\ml(\E_0\times E_\varsigma,\E_1))
\eqn 
due to continuity of the inversion map~\mbox{$B\mapsto B^{-1}$}. We suppose that 
\bqn\label{positivity}
T[u](0,\cdot)\in \ml_+(E_\varsigma,\E_1)\ ,\quad u\in\Sigma_1\ ,
\eqn
that is, $T[u](0,\cdot)$ maps the positive cone $E_\varsigma^+$ into the positive cone $\E_1^+$. Set $$Q(u):=\ell(u)T[u](0,\cdot)\ ,\quad u\in \Sigma_1\ ,$$ and note that from \eqref{A1} and $E_\vartheta\dhr E_\varsigma$ we have
\bqn\label{Q}
Q\in C\big(\Sigma_1,\ml(E_\varsigma,E_\vartheta)\big)\cap C\big(\Sigma_1,\mk(E_\varsigma)\big)\ .
\eqn
We further suppose that
\bqn\label{EV}
\begin{split}
&\lambda_0^{-1}>0\ \text{is a simple eigenvalue of}\ Q(0)\ \text{with eigenvector}\ \Phi_0\in \mathrm{int}(E_\varsigma^+),\\
&\text{and there is no other eigenvalue of $Q(0)$ with an eigenvector in $E_\varsigma^+$}\ .
\end{split}
\eqn
For our global bifurcation results we shall strengthen some of the conditions. We point out that   assumptions above are met in many applications, roughly \eqref{T} is satisfied e.g. for second order operators in divergence form and \eqref{positivity} is the maximum principle for parabolic equations. We shall be more specific in Section~\ref{Sec5}. Regarding assumptions \eqref{comp} and \eqref{EV} we note:

\begin{rems}\label{R1}
(a) Let $a_m<\infty$. If $\alpha\in [0,1)$ and $s\in [0,1-\alpha)$, then $\E_1\dhr W_p^s(J,E_\alpha)$.

\begin{proof}
 This follows from a generalized Aubin-Dubinskii lemma, see \cite[Thm.1.1]{AmannGlasnik00}.
\end{proof}
 
(b) If $Q(0)\in \mk(E_\varsigma)$ is strongly positive, i.e. if $Q(0)\Phi\in \mathrm{int}(E_\varsigma^+)$ for each $\Phi\in  E_\varsigma^+\setminus\{0\}$, then \eqref{EV} holds with $\lambda_0^{-1}$ equals the spectral radius of $Q(0)$.

\begin{proof}
 This is a consequence of the Krein-Rutman theorem \cite[Thm.12.4]{DanersKochMedina}.
\end{proof}

\end{rems}

Let us note that $(\lambda,u)\in\R\times\Sigma_1$ solves \eqref{1}, \eqref{2} if and only if
\bqn\label{sol1}
u=T[u]\big(0,u(0)\big)\ ,\qquad u(0)=\lambda Q(u)u(0)\ ,
\eqn
which follows by plugging the solution of \eqref{1}, given by the first identity of \eqref{sol1}, into \eqref{2}. Equivalently, setting $S:=T[0]$ we see that $(\lambda,u)\in\R\times\Sigma_1$ solves \eqref{1}, \eqref{2} if and only if 
$$u=S\big((\A(0)-\A(u))u \,,\, \lambda\ell(u) u\big)\ .$$
We shall use both characterizations of solutions in the sequel. Note that \eqref{sol1} implies that $u(0)$ (if nonzero) is an eigenvector of $Q(u)$ with eigenvalue $\lambda^{-1}$.

\subsection{Local Bifurcation}

By what we have just observed, solving \eqref{1}, \eqref{2} is equivalent to finding the zeros $(\lambda,u)$ of the function $F:\R\times \Sigma_1\rightarrow \E_1$ defined as
$$
F(\lambda,u):=u-S\big((\A(0)-\A(u))u \,,\, \lambda\ell(u) u\big)\ ,\quad (\lambda,u)\in \R\times\Sigma_1\ .
$$
Let $$\mathfrak{S}:=\{(\lambda,u)\in \R\times \Sigma_1\,;\, F(\lambda,u)=0\}\ .$$ 
Clearly, $(\lambda,u)=(\lambda,0)$ for $\lambda\in\R$ gives a trivial branch in $\mathfrak{S}$ of solutions to \eqref{1}, \eqref{2}. We shall next show that  a nontrivial branch of positive solutions bifurcates from this branch at the point $(\lambda,u)=(\lambda_0,0)$. For this we first show that the Frech\'et derivative $F_u(\lambda,u)$ with respect to $u$, given by
\bqn\label{frechet}
F_u(\lambda,u)[\phi]= \phi-S\big((\A(0)-\A(u))\phi \,,\, \lambda\ell(u)\phi\big)-S\big(-\A_u(u)[\phi] u \,,\, \lambda\ell_u(u)[\phi] u\big)\ ,\quad \phi\in\E_1\ ,
\eqn
is an index zero Fredholm operator. This will follow from the following observation: 

\begin{prop}\label{fredholm}
Let $(\lambda,u)\in \R\times\Sigma_1$ be fixed and set 
$$
\mathcal{F}\phi:=\mathcal{F}({\lambda,u})\phi:= \phi-S\big((\A(0)-\A(u))\phi \,,\, \lambda\ell(u)\phi\big)\ ,\quad \phi\in\E_1\ .
$$
Then $\mathcal{F}\in\ml(\E_1)$ is a Fredholm operator of index zero.  More precisely,
$$
\mathrm{ker}\big(\mathcal{F}\big)=\big\{T[u](0,w)\,;\, w\in \mathrm{ker}\big(1-\lambda Q(u)\big)\big\}
$$
and
\bqn\label{rg}
\mathrm{rg}\big(\mathcal{F}\big)=\big\{h\in\E_1\,;\, h(0)+\lambda\ell (u) T[u]\big(\partial_ah+\A(0)h, 0\big)\in \mathrm{rg}\big(1-\lambda Q(u)\big)\big\}
\eqn
with
$$
\mathrm{dim}\big(\mathrm{ker}(\mathcal{F})\big)=\mathrm{codim}\big(\mathrm{rg}(\mathcal{F})\big)=\mathrm{dim}\big(\mathrm{ker}(1-\lambda Q(u))\big)<\infty\ .
$$
\end{prop}

\begin{proof}
The idea of the proof is the same as in \cite[Lem.2.1]{WalkerSIMA} and it is rather the functional analytic setting that has to be  modified slightly. For the reader's ease we include the proof here:
By definition of $S=T[0]$, the equation $\mathcal{F}\phi=h$ for $\phi, h\in\E_1$ is equivalent to
\begin{align}
\partial_a\phi+\A(u)\phi&=\partial_a h+\A(0)h\ ,\label{L1}\\
\phi(0)-\lambda\ell(u)\phi&=h(0)\ .\label{L2}
\end{align}
From \eqref{L1} it follows  
\bqn\label{u1}
\phi=T[u]\big(\partial_ah+\A(0)h, 0\big)+T[u]\big( 0, \phi(0)\big)
\eqn
 and when plugged into \eqref{L2} we obtain 
\bqn\label{u2}
(1-\lambda Q(u))\phi(0)= h(0)+\lambda\ell (u) T[u]\big(\partial_ah+\A(0)h, 0\big)\ .
\eqn
The statement of the proposition is trivial if $\lambda=0$, so let $\lambda\not= 0$.
If $1/\lambda$ belongs to the resolvent set of $Q(u)\in\mathcal{K}(E_\varsigma)$, then \eqref{u1}, \eqref{u2} entail a trivial kernel $\kk(\mathcal{F})=\{0\}$. Moreover, in this case, for an arbitrary $h\in \E_1$, there is a unique $\phi(0)\in E_\varsigma$ solving \eqref{u2} as its right hand side  belongs to $E_\varsigma$. Consequently the corresponding $\phi\in\E_1$ given by \eqref{u1} is the unique solution to $\mathcal{F}\phi=h$. This  gives the assertion in this case.

Otherwise, if $1/\lambda$ is an eigenvalue of $Q(u)\in\mathcal{K}(E_\varsigma)$, then \eqref{u1}, \eqref{u2} yield the characterization of $\kk(\mathcal{F})$ and $\rg(\mathcal{F})$ as claimed. In particular, 
since $T[u]$ is an isomorphism, we deduce $\mathrm{dim}(\kk(\mathcal{F}))=\mathrm{dim}(\kk(1-\lambda Q(u)))$ which is a finite number because $1/\lambda$ is an eigenvalue of the compact operator $Q(u)$. Moreover, $\rg(\mathcal{F})$ is closed in $\E_1$ since $M:=\rg(1-\lambda Q(u))$ is closed by the compactness of $\lambda Q(u)$ and due to \eqref{emb}, \eqref{A1}, and \eqref{T}.
Next, to compute $\codim(\rg(\mathcal{F}))$ note that $$\codim(M)=\dimens(\kk(1-\lambda Q(u)))<\infty\ ,$$ hence $M$ is complemented in $E_\varsigma$ which yields a direct sum decomposition $E_\varsigma=M\oplus N$. Denoting by \mbox{$P_M\in\ml(E_\varsigma)$} a projection onto $M$ along $N$, we set
\bqn\label{18B}
\mathbb{P}h:=S\big(\partial_a h+\A(0)h\,,\, P_Mh (0)-(1-P_M)\lambda\ell (u) T[u]\big(\partial_ah+\A(0)h, 0\big)\big)\ ,\quad h\in\E_1\ ,
\eqn
 and obtain $\mathbb{P}\in\mathcal{L}(\E_1)$ from \eqref{emb}, \eqref{A1}, and \eqref{T}. Since
$$
\big(\partial_a +\A(0)\big)(\mathbb{P}h)=\partial_a h+\A(0)h
\ ,\quad \gamma_0(\mathbb{P}h)=  P_Mh (0)-(1-P_M)\lambda\ell (u) T[u]\big(\partial_ah+\A(0)h, 0\big)\ ,
$$
the characterization \eqref{rg} actually implies that $\mathbb{P}$ maps $\E_1$ into $\rg(\mathcal{F})$. 
Furthermore, if $h\in \rg(\mathcal{F})$, then \eqref{rg} also ensures  $$\mathbb{P}h=S(\partial_ah +\A(0)h,h(0))=h\ ,$$ so  $\mathbb{P}(\rg(\mathcal{F}))=\rg(\mathcal{F})$. Thus $\mathbb{P}^2=\mathbb{P}$ with $\rg(\mathbb{P})=\rg(\mathcal{F})$ is a projection and 
\bqn\label{P1}
\E_1= \kk (\mathbb{P})\oplus \rg(\mathcal{F})\ .
\eqn
Since $S$ is an isomorphism, we obtain 
\bqn\label{P2}
\kk (\mathbb{P})=\{h\in\E_1\,;\, \partial_ah +\A(0)h=0\,,\, h(0)\in N\}\ ,
\eqn
from which we deduce $\mathrm{dim}(\kk (\mathbb{P}))=\mathrm{dim}(N)$ and the statement follows.
\end{proof}

For future purposes let us explicitly state the following decomposition of $\E_1$.

\begin{rem}\label{decomposition}
The direct sum decomposition $$\E_1=\mathrm{span}\big(S(0,\Phi_0)\big)\oplus \rg\big(F_u(\lambda_0,0)\big)$$ holds.
\end{rem}

\begin{proof}
Taking  $(\lambda,u)=(\lambda_0,0)$ in Proposition~\ref{fredholm} and noticing that $\mathcal{F}(\lambda_0,0)=F_u(\lambda_0,0)$, the assertion follows from \eqref{P1} and \eqref{P2} with $N=\kk(1-\lambda_0Q(0))=\R\cdot \Phi_0\ .$
\end{proof}

\begin{cor}\label{fred}
For each~$(\lambda,u)\in \R\times\Sigma_1$, the Frech\'et derivative $F_u(\lambda,u)\in  \ml(\E_1)$ is a Fredholm operator of index zero.
\end{cor}

\begin{proof}
Set 
$$
K(u)\phi:=S\big(-\A_u(u)[\phi] u \,,\, \lambda\ell_u(u)[\phi] u\big)\ ,\quad \phi\in\E_1\ .
$$
Then, by \eqref{comp} and \eqref{A1}, $K(u)$ coincides with the  Frech\'et derivative 
$$
K(u)=D_w S\big(-\A(w)u \,,\, \lambda\ell(w) u\big)\big\vert_{w=u}\in\ml(\F,\E_1)\subset \mk(\E_1)\ .
$$
Consequently, Proposition~\ref{fredholm} and \eqref{frechet} show that $F_u(\lambda,u)$ is a compact perturbation of a Fredholm operator of index zero, so $F_u(\lambda,u)$ itself is a  Fredholm operator of index zero.
\end{proof}

Next, we verify that we may apply the Crandall-Rabinowitz theorem on local bifurcation  for the map $F$.

\begin{cor}\label{CR}
The kernel of  $F_u(\lambda_0,0)$ is one-dimensional, i.e. $\kk(F_u(\lambda_0,0))=\R\cdot S(0,\Phi_0)$, and the transversality condition
$F_{\lambda,u}(\lambda_0,0)\big[1,S(0,\Phi_0)\big]\not\in\rg(F_u(\lambda_0,0))$ is satisfied.
\end{cor}

\begin{proof}
It readily follows from \eqref{EV}, \eqref{frechet}, and Proposition~\ref{fredholm} that $\kk(F_u(\lambda_0,0))=\R\cdot S(0,\Phi_0)$. Moreover,  \eqref{EV} and \eqref{frechet} imply $F_{\lambda,u}(\lambda_0,0)\big[1,S(0,\Phi_0)\big]=-(0,\lambda_0^{-1}\Phi_0)$. Suppose that $$-(0,\lambda_0^{-1}\Phi_0)\in\rg(F_u(\lambda_0,0))\ .$$ Then $\lambda_0^{-1}\Phi_0\in\rg(1-\lambda_0Q(0))$ due to Proposition~\ref{frechet} contradicting the fact that $$\rg(1-\lambda_0Q(0))\cap \kk(1-\lambda_0Q(0))=\{0\}$$ since $\lambda_0^{-1}$ is a simple eigenvalue of $Q(0)$ according to \eqref{EV}.
\end{proof}

Based on the foregoing observations we are in a position to apply the celebrated  Crandall-Rabinowitz theorem \cite{CrandallRabinowitz} on local bifurcation and obtain a branch in $\mathfrak{S}$ of positive solutions to \eqref{1}, \eqref{2}. The following result has  been observed in~\cite{WalkerSIMA}.

\begin{thm}\label{locbif}
Suppose \eqref{densecompact}, \eqref{positiv}, \eqref{comp}, \eqref{A2}, \eqref{positivity}, and \eqref{EV} hold.
Then there are $\ve>0$ and a continuous function \mbox{$(\bar{\lambda},\bar{u}): (-\ve,\ve)\rightarrow \R\times \Sigma_1$} such that the curves $$\mathfrak{K}^\pm:=\{(\bar{\lambda}(t),\bar{u}(t))\,;\, 0\le \pm t<\ve\}\subset \mathfrak{S}$$  bifurcate from the trivial branch $\{(\lambda,0)\,;\,\lambda\in\R\}$ at $(\bar{\lambda}(0),\bar{u}(0))=(\lambda_0,0)$ and 
\bqn\label{ps}
\bar{u}(t)=t S(0,\Phi_0)+o(t)\quad \text{as}\ t\rightarrow 0\ .
\eqn 
Near the bifurcation point $(\lambda_0,0)$, all nontrivial zeros of $F$ lie on the curve $\mathfrak{K}^-\cup\mathfrak{K}^+$.
Moreover, $$\mathfrak{K}^+\setminus\{(\lambda_0,0)\}\subset (0,\infty)\times\dot{\E}_1^+$$
and $\mathfrak{K}^-\cap (0,\infty)\times\dot{\E}_1^+ = \emptyset$.
\end{thm}

\begin{proof}
According to Corollary~\ref{fred}, Corollary~\ref{CR}, and \cite{CrandallRabinowitz} we only have to prove  the  statements on positivity of the last sentence. From \eqref{ps} and \eqref{EV} it follows that $$t^{-1}\gamma_0 \bar{u}(t)=\Phi_0+\gamma_0\frac{o(t)}{t}\in \mathrm{int}(E_\varsigma^+)\quad \text{as}\ t\rightarrow 0\ ,$$ whence $\gamma_0 \bar{u}(t)\in E_\varsigma^+$ provided $t\in (0,\ve)$ is sufficiently small.
Since $$\bar{u}(t)=T[\bar{u}(t)]\big(0,\gamma_0 \bar{u}(t)\big)\ ,$$ we conclude $\bar{u}(t)\in \dot{\E}_1^+$ from \eqref{positivity} and $\bar{\lambda}(t)>0$ for $t\in (0,\ve)$  from the assumption $\lambda_0>0$. Consequently, $\mathfrak{K}^+\setminus\{(\lambda_0,0)\}\subset (0,\infty)\times\dot{\E}_1^+$. The same argument shows that  $\mathfrak{K}^-\cap (0,\infty)\times\dot{\E}_1^+ $ is empty.
\end{proof}

Note that if $\A$ and $\ell$ in \eqref{A2} are real analytic, then so is the curve $(\bar{\lambda},\bar{u}): (-\ve,\ve)\rightarrow \R\times \Sigma_1$    since $F:\R\times \Sigma_1\rightarrow \E_1$ is real analytic in this case. 

The direction of bifurcation may be determined from the second  relation \eqref{sol1}. For instance, if \mbox{$Q(u)\in\mk(E_\varsigma)$} is strongly positive for  $u\in\Sigma_1$, then \eqref{sol1} implies $\lambda r(Q(u))=1$ for any positive solution $(\lambda,u)\in\R\times\Sigma_1$ of \eqref{1}, \eqref{2} according to the Krein-Rutman theorem. Conditions may then be imposed which ensure 
$r(Q(u))\le 1$ so that bifurcation is necessarily supercritical. For further details we refer to \cite{WalkerSIMA}.\\

 We next prove a compactness property of the solution set $\mathfrak{S}$ that we shall use in the coming subsections.

\begin{lem}\label{compact}
Any bounded and closed subset of $\mathfrak{S}$ is compact in $\R\times \E_1$.
\end{lem}

\begin{proof}
Let $(\lambda_n,u_n)_{n\in\N}$ be any sequence in a closed subset of $\mathfrak{S}$ such that $\|(\lambda_n,u_n)\|_{\R\times \E_1}\le c_0$ for all $n\in\N$ and some $c_0\in\R$. Then 
\bqn\label{120}
u_n=T[u_n]\big(0,u_n(0)\big)\ ,\qquad u_n(0)=\lambda_n Q(u_n)u_n(0)\ .
\eqn
Due to \eqref{comp}, we may assume without loss of generality that $(\lambda_n,u_n)\rightarrow (\lambda,u)$ in $\R\times \Sigma$. According to \eqref{A1} and \eqref{Q} this implies $T[u_n]\big(0,\cdot\big)\rightarrow T[u]\big(0,\cdot\big)$ in $\ml(E_\varsigma,\E_1)$ and $Q(u_n)\rightarrow Q(u)$ in $\ml(E_\varsigma,E_\vartheta)$. Also note that \eqref{emb} entails $\|u_n(0)\|_{E_\varsigma}\le c$ for $n\in\N$. Consequently, from \eqref{120} we derive
\bqn\label{est}
\|u_n(0)\|_{E_\vartheta}\le \vert\lambda_n\vert\, \|Q(u_n)\|_{\ml(E_\varsigma,E_\vartheta)}\, \|u_n(0)\|_{E_\varsigma}\, \le \, c\ ,\qquad n\in\N\ ,
\eqn
and we thus may assume without loss of generality that $u_n(0)\rightarrow v$ in $E_\varsigma$ since $E_\vartheta\dhr E_\varsigma$. Now \eqref{120} shows that $v=\lambda Q(u)v$ and $u_n=T[u_n]\big(0,u_n(0)\big)\rightarrow T[u]\big(0,v\big)$ in $\E_1$. Clearly, $u=T[u]\big(0,v\big)$ since $u_n\rightarrow u$ in $\F$ and we conclude that $(\lambda_n,u_n)\rightarrow (\lambda,u)$ in $\R\times\E_1$. This proves the assertion.
\end{proof}

\section{Global Continua}\label{Sec4}	

We shall show that the local curve $\mathfrak{K}^+$ provided by Theorem~\ref{locbif} is contained in a global continuum of positive solutions to \eqref{1}, \eqref{2}. In this section we make use of the unilateral global bifurcation theory in the spirit of Rabinowitz's alternative \cite{Rabinowitz,LopezGomez} as  proposed in~\cite{ShiWang}. In Section~\ref{Sec3} we shall give a slightly different approach by means of analytic bifurcation theory~\cite{BuffoniToland}.

In order to apply the results of \cite{ShiWang} we have to strengthen certain conditions. More precisely, in the following we suppose in addition to the assumptions stated in Section~\ref{Sec2} that
\bqn\label{reflexive}
E_0'\ \text{and } E_1 \text{ are separable }
\eqn
and we strengthen assumption \eqref{A2} to a convexity condition
\bqn\label{A2convex}
\big(\partial_a+[(1-\alpha)\A(0)+\alpha\A(u)],\gamma_0\big)\in\Isom(\E_1,\E_0\times E_\varsigma\big)\ ,\quad u\in \Sigma_1\ ,\quad \alpha\in [0,1]\ .
\eqn 
We also assume that
\bqn\label{l0}
\ell(0)\in\ml_+(\E_1,E_\varsigma)
\eqn
and
\bqn\label{Q0}
Q(0)\in\mk(E_\varsigma)\text{ is strongly positive } .
\eqn
Recall that this last assumption  implies \eqref{EV} with $\lambda_0^{-1}$ being given by the spectral radius of $Q(0)$ according to Remark~\ref{R1}.
Note that $E_0$ is separable since $E_1$ is separable and dense in $E_0$. Also note that if $E_0$ is reflexive (and separable), then $E_0'$ is separable. \\

Assumption~\eqref{reflexive} implies the following remark which is needed in order to apply the result of~\cite{ShiWang}. 

\begin{lem}\label{restrepo}
$\E_1$ can be equipped with an equivalent norm which is differentiable at any point different from 0.
\end{lem}

\begin{proof}
Due to \cite{Restrepo}, as $\E_1$ is  separable since $E_1$ is, the statement that $\E_1$ can be equipped with an equivalent norm which is differentiable at any point different from 0 is equivalent to saying that the dual space $\E_1'$ is separable.
But since $E_0'$ is separable, the space $\E_0'=L_{p'}(J,E_0')$ (with $1=1/p+1/p'$) is also separable and, moreover, densely injected in $\E_1'$ since $\E_1$ is densely injected in $\E_0$. Whence $\E_1'$ is separable.
\end{proof}

The convexity condition \eqref{A2convex} yields:

\begin{lem}\label{Fredconvex}
For $(\lambda,u)\in \R\times\Sigma_1$ and $\alpha\in[0,1]$, the operator $$(1-\alpha)F_u(\lambda,0)+\alpha F_u(\lambda,u)\in\ml(\E_1)$$ is Fredholm of index zero.
\end{lem}

\begin{proof}
Let $(\lambda,u)\in \R\times\Sigma_1$ and $\alpha\in[0,1]$.
Noticing that
\bqnn
\begin{split}
(1-\alpha)F_u(\lambda,0)[\phi]+\alpha F_u(\lambda,u)[\phi]=\ & \phi-S\big(\alpha(\A(0)-\A(u))\phi \,,\, \lambda[(1-\alpha)\ell(0)+\alpha\ell(u)]\phi\big)\\
& -\alpha S\big(-\A_u(u)[\phi] u \,,\, \lambda\ell_u(u)[\phi] u\big)
\end{split}
\eqnn
for $\phi\in\E_1$, the proof is the same as in Proposition~\ref{fredholm} and Corollary~\ref{fred} by taking \eqref{A2convex} into account.
\end{proof}

Now we can prove that there is  global continuum of positive solutions to \eqref{1}, \eqref{2}. 

\begin{thm}\label{T_continua}
Suppose \eqref{densecompact}, \eqref{positiv}, \eqref{comp}, \eqref{A1}, \eqref{positivity}, and \eqref{reflexive}--\eqref{Q0} hold. Then there is a connected component $\mathfrak{C}^+$ of $\bar{\mathfrak{S}}$ containing the branch $\mathfrak{K}^+$ such that $\mathfrak{C}^+\setminus\{(\lambda_0,0)\}\subset (0,\infty)\times\dot{\E}_1^+$. 
Moreover, one of the alternatives
\begin{itemize}
\item[(i)] $\mathfrak{C}^+$ is unbounded in $\R\times\E_1$, or
\item[(ii)] $\mathfrak{C}^+$ intersects with the boundary $\R\times\partial\Sigma_1$
\end{itemize}
occurs. In particular, if $\Sigma=\F$, then $\mathfrak{C}^+$ is unbounded in $\R\times\E_1$.
\end{thm}

\begin{proof}
Due to Lemma~\ref{restrepo} and Lemma~\ref{Fredconvex}, we may apply \cite[Thm.4.4]{ShiWang} and deduce that $\mathfrak{K}^+$ is contained in a connected component  $\mathfrak{C}^+$ of $\bar{\mathfrak{S}}$ and one
of the alternatives
\begin{itemize}
\item[(a)] $\mathfrak{C}^+$ is not compact in $\R\times\Sigma_1$, or
\item[(b)] $\mathfrak{C}^+$ contains a point $(\lambda_*,0)$ with $\lambda_*\not=\lambda_0$, or
\item[(c)] $\mathfrak{C}^+$ contains a point $(\lambda,z)$ with $z\not=0$ and $z\in \rg\big(F_u(\lambda_0,0)\big)$ 
\end{itemize}
occurs, where we have used in (c) the decomposition $\E_1=\mathrm{span}\big(S(0,\Phi_0)\big)\oplus \rg\big(F_u(\lambda_0,0)\big)$ stated in Remark~\ref{decomposition}. Notice that, owing to Lemma~\ref{compact} (see \cite[Rem.4.2]{ShiWang}), alternative (a) is equivalent to saying that 
\begin{itemize}
\item[(i)] $\mathfrak{C}^+$ is unbounded in $\R\times\E_1$, or 
\item[(ii)] $\mathfrak{C}^+$ intersects with the boundary $\R\times\partial\Sigma_1$.
\end{itemize}
According to Theorem~\ref{locbif}, the component $\mathfrak{C}^+$ near the bifurcation point $(\lambda_0,0)$ coincides with~$\mathfrak{K}^+$. Next, we show that $\mathfrak{C}^+\setminus\{(\lambda_0,0)\}\subset (0,\infty)\times\dot{\E}_1^+$. Indeed, if $\mathfrak{C}^+$ leaves $(0,\infty)\times\dot{\E}_1^+$ at some point $(\lambda,u)\in\mathfrak{C}^+\cap \R\times\E_1$ with $(\lambda,u)\notin (0,\infty)\times\dot{\E}_1^+$, there is a sequence $((\lambda_j,u_j))_{j\in\N}$ in $\mathfrak{C}^+\cap (0,\infty)\times\dot{\E}_1^+$  such that $(\lambda_j,u_j)\rightarrow (\lambda,u)$ in $\R\times\E_1$.  Clearly, $\lambda\ge 0$ and $u\in \E_1^+$ with $\lambda=0$ or $u\equiv 0$. But since $(\lambda,u)\in\mathfrak{S}$, we readily deduce from \eqref{sol1} that $\lambda=0$ implies $u\equiv 0$. Hence $u\equiv 0$ in any case, i.e. $(\lambda_j,u_j)\rightarrow (\lambda,0)$ in $\R\times\E_1$. Again by \eqref{sol1}, we have 
\bqn\label{sol111}
u_j=T[u_j]\big(0,u_j(0)\big)\ ,\qquad u_j(0)=\lambda_j Q(u_j)u_j(0)\ .
\eqn
Since $v_j:=u_j/\|u_j\|_{\E_1}$ defines a bounded sequence in $\E_1$, by \eqref{comp} we may  extract a subsequence of $(v_j)$ (which we do not index)  which converges to some $v$ in $\F$. From \eqref{T} and \eqref{Q} we deduce $T[u_j]\rightarrow T[0]=S$ in $\ml(\E_0\times E_\varsigma,\E_1)$ and $Q(u_j)\rightarrow Q(0)$ in $\ml (E_\varsigma,E_\vartheta)$. As in \eqref{est} 
we then obtain from \eqref{Q} and \eqref{emb} that 
$$
\|u_j(0)\|_{E_\vartheta}\le c\|u_j(0)\|_{E_\varsigma} \le c \|u_j\|_{\E_1}\ ,
$$
from which we conclude that the sequence $(v_j(0))$ is bounded in $E_\vartheta\dhr E_\varsigma$. So, extracting a further subsequence (which we again do not index) we see that $v_j(0)\rightarrow w$ in $E_\varsigma^+$. Letting $j\rightarrow\infty$ in  \eqref{sol111} yields
$$
v=S\big(0,w\big)\ ,\qquad w=\lambda  Q(0)w\ ,
$$
from which we first deduce that $\lambda>0$ since otherwise $w=0$ implying the contradiction \mbox{$v\equiv 0$}. Consequently, $w\in E_\varsigma^+$ is an eigenvector of $Q(0)$ to the eigenvalue $1/\lambda$. Thus \mbox{$\lambda=\lambda_0$} and \mbox{$w=\alpha\Phi_0$} for some $\alpha>0$ according to \eqref{EV}, hence $(\lambda,u)=(\lambda_0,0)$.
Therefore, $\mathfrak{C}^+$ leaves the set $(0,\infty)\times\dot{\E}_1^+$ only at $(\lambda_0,0)$ and \mbox{$\mathfrak{C}^+\setminus\{(\lambda_0,0)\}$} is thus contained in $ (0,\infty)\times\dot{\E}_1^+$. In particular, alternative (b) above does not occur. We finally show that alternative (c) does not occur as well. Suppose to the contrary that $\mathfrak{C}^+$ contains a point $(\lambda,z)$ with $z\not=0$ and $z=F_u(\lambda_0,0)\zeta$ for some $\zeta\in\E_1$. Then $z\in\dot{\E}_1^+$ and $\zeta-z=S(0,\lambda_0\ell(0)\zeta)$. Recall from Corollary~\ref{CR} that $\phi_*:=S(0,\Phi_0)$ with $\Phi_0$ from \eqref{EV} satisfies $\phi_*=S(0,\lambda_0\ell(0)\phi_*)$. Since $\Phi_0\in\mathrm{int}(E_\varsigma^+)$, we find $\kappa>0$ such that $\kappa\Phi_0+\zeta(0)-z(0)$ belongs to $E_\varsigma^+$. Set $\psi:=\kappa \phi_*+\zeta-z$. Due to \mbox{$\psi=\lambda_0 S(0,\ell(0)(\psi+z))$} we conclude  $\partial_a\psi+\A(0)\psi=0$ on the one hand from which $\psi=  S(0,\psi(0))$, and $$\psi(0)=\lambda_0\ell(0)\psi+\lambda_0\ell(0) z$$ on the other. Combining these two observations we derive the equation \bqn\label{psot}
(1-\lambda_0 Q(0))\psi(0)=\lambda_0\ell(0) z \ .
\eqn
However, since $\lambda_0^{-1}$ is the spectral radius of the strongly positive compact operator $Q(0)$ and since $\lambda_0\ell(0) z\in E_\varsigma^+$ by \eqref{Q0}, equation \eqref{psot} has no positive solution according to \cite[Cor.12.4]{DanersKochMedina} contradicting $\psi(0)=\kappa\Phi_0+\zeta(0)-z(0)\in E_\varsigma^+$. Therefore, alternative (c) above is impossible and the theorem is proven.
\end{proof}

\section{Global Branches in the Analytic Case}\label{Sec3}

In this section we shall show that   assumptions \eqref{reflexive}-\eqref{Q0} of Theorem~\ref{T_continua} are not needed to
 extend the function $(\bar{\lambda},\bar{u})$  from $(0,\ve)$ to $(0,\infty)$ provided that $\A$ and $\ell$ are real analytic. In this case we obtain a slightly better result. So, let us strengthen assumption \eqref{A1} to
\bqn\label{A1a}
\A\in C^\omega\big(\Sigma, \ml(\E_1,\E_0)\big)\ ,\qquad \ell \in C^\omega\big(\Sigma, \ml(\E_1,E_\vartheta)\big)\ .
\eqn
As noted in Section~\ref{Sec2}, the function $(\bar{\lambda},\bar{u})$ is real analytic in this case. 

\begin{thm}\label{T-analytic}
Suppose \eqref{densecompact}, \eqref{positiv}, \eqref{comp}, \eqref{A2}, \eqref{positivity}, \eqref{EV}, and \eqref{A1a} hold. Then there is a  continuous curve $\mathfrak{R}^+=\{(\bar{\lambda}(t),\bar{u}(t))\,;\, t\in [0,\infty)\}\subset \mathfrak{S}$ extending $\mathfrak{K}^+$. The curve $\mathfrak{R}^+\setminus \{(\lambda_0,0)\}$ lies in \mbox{$(0,\infty)\times \dot{\E}_1^+$} and has at each point a local analytic and injective reparametrization. Moreover, one of the alternatives
\begin{itemize}
\item[(i)] $\|(\bar{\lambda}(t),\bar{u}(t))\|_{\R\times\E_1}\rightarrow \infty$ as $t\rightarrow$, or
\item[(ii)] $\bar{u}(t)\rightarrow \partial\Sigma_1$ as $t\rightarrow\infty$
\end{itemize}
occurs. In particular, if $\Sigma=\F$, then (i) occurs.
\end{thm}

\begin{proof}
Due to Corollary~\ref{fred}, Corollary~\ref{CR}, and Lemma~\ref{compact} we may apply \cite[Thm.9.1.1]{BuffoniToland}. Consequently, there is a  continuous curve $\mathfrak{R}^+=\{(\bar{\lambda}(t),\bar{u}(t))\,;\, t\in [0,\infty)\}\subset \mathfrak{S}$ extending $\mathfrak{K}^+$ and having a local analytic and injective reparametrization. For this curve $\mathfrak{R}^+$,  one of the alternatives
\begin{itemize}
\item[(i)] $\|(\bar{\lambda}(t),\bar{u}(t))\|_{\R\times\E_1}\rightarrow \infty$ as $t\rightarrow$, or
\item[(ii)] $\bar{u}(t)\rightarrow \partial\Sigma_1$ as $t\rightarrow\infty$, or
\item[(iii)] $\mathfrak{R}^+$ is a closed loop, i.e. there is a minimal $\tau>0$  such that \mbox{$\mathfrak{R}^+=\{(\bar{\lambda}(t),\bar{u}(t))\,;\, 0\le t\le\tau\}$} and $(\bar{\lambda}(\tau),\bar{u}(\tau))=(\bar{\lambda}(0),\bar{u}(0))=(\lambda_0,0)$,
\end{itemize}
occurs. Moreover, if $(\bar{\lambda}(t_1),\bar{u}(t_1))=(\bar{\lambda}(t_2),\bar{u}(t_2))$ for some $t_1\not= t_2$ with $$\kk (F_u(\bar{\lambda}(t_1),\bar{u}(t_1)))=\{0\}\ ,$$ then (iii) occurs and $\vert t_1-t_2\vert$ is an integer multiple of $\tau$. Finally, the set $$\{t\ge 0\,;\, \kk(F_u(\bar{\lambda}(t ),\bar{u}(t )))\not=\{0\}\}$$ has no accumulation points. So, the assertion   follows provided we can prove  that the curve \mbox{$\mathfrak{R}^+\setminus \{(\lambda_0,0)\}$} lies in $(0,\infty)\times \dot{\E}_1^+$ and that alternative (iii) does not occur. For this we use an argument similar to \cite[Thm.9.2.2]{BuffoniToland}: Set 
$$
T_*:=\sup\,\{T>0\,;\, (\bar{\lambda}(t),\bar{u}(t))\in (0,\infty)\times \dot{\E}_1^+ \ \text{for}\ 0<t<T\}
$$ 
and note that $T_*\ge \ve$ according to Theorem~\ref{locbif}.  Assuming $T_*<\infty$, there is a sequence $(t_j)$ with $t_j\nearrow T_*$ such that $(\lambda_j,u_j):=(\bar{\lambda}(t_j),\bar{u}(t_j))\in (0,\infty)\times \dot{\E}_1^+$ converges to $(\lambda,u):=(\bar{\lambda}(T_*),\bar{u}(T_*))$ in $\R\times\E_1$ and $(\lambda,u)\not\in (0,\infty)\times \dot{\E}_1^+$. But then, the same argument as in the proof of Theorem~\ref{T_continua} yields $(\bar{\lambda}(T_*),\bar{u}(T_*))=(\lambda_0,0)$. Since Theorem~\ref{locbif} implies that the bifurcation curve which lies in $\R^+\times\E_1^+$ and passes through $(\lambda_0,0)$ is near this point uniquely determined by $\mathfrak{K}^+$, we derive that $(\bar{\lambda}(t),\bar{u}(t))$ belongs to $\mathfrak{K}^+$ for~$t$ less, but close to $T_*$. Consequently, there are sequences $r_k\searrow 0$ and $s_k\searrow 0$ such that $(\bar{\lambda}(r_k),\bar{u}(r_k))=(\bar{\lambda}(T_*-s_k), \bar{u}(T_*-s_k))$ and $\kk\big(F_u(\bar{\lambda}(r_k),\bar{u}(r_k))\big)=\{0\}$. But then, as stated above, the minimally chosen $\tau>0$ with $(\lambda_0,0)=(\bar{\lambda}(\tau),\bar{u}(\tau))$ divides $T_*-s_k-r_k$ for each $k\in\N$ which is obviously impossible. Therefore, $T_*=\infty$ and alternative (iii) above does not occur.
\end{proof}

Note that it is not claimed in Theorem~\ref{T-analytic} that  $\mathfrak{R}^+$ is a maximal connected subset of  $\mathfrak{S}$. Other curves or manifolds in  $\mathfrak{S}$ may intersect  $\mathfrak{R}^+$. We also point out that alternative (i) in Theorem~\ref{T-analytic} is stronger than saying that $\mathfrak{R}^+$ is unbounded in $\R\times\E_1$ (see Theorem~\ref{T_continua}). 

\section{Example}\label{Sec5}

Let us consider an example.
Let $u=u(a,x)$ denote the distribution density of individuals of a population with age $a\in J:= [0,a_m)$ 
at spatial position $x$ in a bounded, sufficiently smooth space region $\Om\subset\R^n$, where $a_m\in (0,\infty)$ denotes the maximal age.
Suppose that the individuals move within $\Om$ and that dispersal speed $d>0$ depends
on the local overall population; that is, suppose that movement is governed by a density-dependent diffusion 
term $-\mathrm{div}_x\big(d(U)\nabla_xu\big)$, where 
\bqn\label{D}
U(x):=\int_0^{a_m}u(a,x)\rd a
\eqn 
is the overall 
population at spatial position $x\in\Om$. Assume further that individuals cannot leave the space region $\Om$ so 
that the behavior on the boundary $\partial\Om$ is described by a Neumann condition $\partial_\nu u=0$, with $\nu$ 
denoting the outward unit normal to $\partial\Om$. For a given  distribution~$u$ let $\mu(U,\cdot)$ and $\lambda b(U,\cdot)$ denote the death rate respectively the parameter-dependent birth rate. Then equilibrium
(i.e.,  time-independent) solutions to the corresponding evolution problem satisfy the following equations:
\begin{align}
&\partial_a u-\mathrm{div}_x\big(d(U(x))\nabla_xu\big)+\mu(U(x),a)u=0\,, && a\in (0,a_m)\,, & x\in\Om\,, 
\label{Aa}\\
&u(0,x)=\lambda\int_0^{a_m} b(U(x),a)\, u(a,x)\,\rd a\,, &&& x\in\Om\,, 
\label{B}\\
&\partial_\nu u(a,x)= 0\,, && a\in (0,a_m)\,, & x\in\partial\Om\, .
\label{C}
\end{align}
Fix $p\in (n+2,\infty)$ and set 
$$E_1:=\Wpb^2(\Om):=\{v\in W_p^2(\Om); \partial_\nu v=0\}\dhr E_0:=L_p(\Om)
$$
and
$$
\E_1:=L_p(J,\Wqb^2(\Om))\cap W_p^1(J,L_p(\Om))\ ,\qquad \E_0:=L_p(J,L_p(\Om))\ .
$$%
 Clearly, \eqref{densecompact} and \eqref{reflexive} hold.
Observe then that the interpolation result  \cite[Thm.~\!4.3.3]{Triebel} and Sobolev's embedding theorem imply
\bqn
\label{alpha}
E_1=\Wqb^2(\Om)\hookrightarrow E_\varsigma:=\big(L_p(\Om),\Wqb^2\big)_{1-1/p,p}\doteq \Wqb^{2(1-1/p)} \hookrightarrow C^1(\bar{\Om})\ . 
\eqn
Thus $\mathrm{int}(E_\varsigma^+)\not=\emptyset$ while Remark~\ref{R1} implies that $$\E_1\dhr \F:=W_p^s(J,E_\vartheta)$$ for some $\vartheta>1-1/p=\varsigma$ and some $s\in (0,1/p)$, where $E_\vartheta\doteq \Wpb^{2\vartheta}(\Om)$.
Let \mbox{$d\in C^4(\R)$} satisfy $d(z)\ge\underline{d}>0$ for $z\in \R$ and let $\mu, b\in C^4\big(\R\times [0,a_m],\R^+\big)$ with $b(0,\cdot)\not= 0$. 
For \mbox{$u\in \F\hookrightarrow L_1(J,E_\vartheta)$}, set $$U:=\int_0^{a_m}u(a,\cdot)\rd a\in E_\vartheta$$ and define
$$
\A(u,a)w:=-\mathrm{div}_x\big(d(U)\nabla_xw\big) +\mu(U,a) w\ , \quad w\in E_1\ ,\quad u\in\F\ ,\quad a\in J\ .
$$
Then, by \cite[Prop.4.1]{WalkerAMPA}, 
$$
\A\in C^1\big(\F,L_\infty(J,\ml(\Wqb^2(\Om),L_p(\Om)))\big)
$$
and thus in particular $\A\in C^1(\F,\ml(\E_1,\E_0))$. Since  \cite[Prop.4.1]{WalkerAMPA} together with the multiplication result of \cite[Thm.4.1]{AmannMultiplication} ensure $\ell \in C^1\big(\F, \ml(\E_1,E_\vartheta)\big)$ as well with
\bqnn
\lambda\ell(v)u:=\int_0^{a_m} \lambda b(v,a) u(a)\rd a\ ,\quad v\in\F\ ,\quad u\in\E_1\ ,
\eqnn

condition \eqref{A1} holds.
Note that for $\alpha\in [0,1]$, $u\in\F$ and $w\in E_1$ we have
 \bqnn
\begin{split}
\A_\alpha(u,\cdot)w:&=(1-\alpha)\A(0,\cdot)w+\alpha\A(u,\cdot)w\\
&=-\mathrm{div}_x\big([(1-\alpha)d(0)+\alpha d(U)]\nabla_xw\big) +[(1-\alpha)\mu(0,\cdot)+\alpha\mu(U,\cdot)] w
\end{split}
\eqnn
with $(1-\alpha)d(0)+\alpha d(U)\ge \underline{d}$. Hence, for $\alpha\in [0,1]$, $u\in\F$, and $a\in J$  the operator $-\A_\alpha(u,a)$ is resolvent positive, generates a contraction semigroup on each $L_q(\Om)$, $1<q<\infty$ (see \cite{AmannIsrael}), and is self-adjoint in $L_2(\Om)$. Hence \cite[III.Ex.4.7.3,III.Thm.4.10.10]{LQPP} entail \eqref{A2convex}. Since for $u\in\F$ fixed, the mapping $\A(u,\cdot):[0,a_m]\rightarrow \ml(\Wqb^2(\Om),L_p(\Om))$ is H\"older continuous, there is a unique positive evolution operator $\Pi_u(a,\sigma)$, $0\le \sigma\le a\le a_m$ on $E_0$ corresponding to $\A(u,\cdot)$, see \cite[II.Cor.4.4.2.,II.Thm.6.4.2]{LQPP}. In particular, $T[u](0,\cdot)=\Pi_u(\cdot,0)\in\ml_+(E_\varsigma,\E_1)$ for $u\in\F$ and
$\ell(0)\in\ml_+(\E_1,E_\varsigma)$,
that is, \eqref{positivity} and \eqref{l0} hold. Also note that the maximum principle ensures that $\Pi_0(a,0)\in\mk(E_\varsigma)$ is strongly positive for each $a\in J\setminus\{0\}$. As $b(0,\cdot)\not= 0$ we conclude (see \cite[Sect.3]{WalkerSIMA}) that 
$$
Q(0)=\int_0^{a_m}b(0,a)\Pi_0(a,0) \,\rd a\in\mk(E_\varsigma)
$$
is strongly positive, whence \eqref{Q0}. Consequently, we are in a position to apply Theorem~\ref{T_continua} and deduce that there is an unbounded continuum of positive solutions $(\lambda,u)$ in $(0,\infty)\times\E_1^+$ to~\mbox{\eqref{Aa}-\eqref{C}}.

Let us point out that this is just one simple example which can be extended in various ways. For instance, one may consider more general (uniformly elliptic) differential operators that depend also locally on age $a$. Also the regularity assumptions  are not chosen optimally and the phase space $E_0$ can be any $L_q(\Om)$ provided $q\in (1,\infty)$. In Theorem~\ref{T-analytic} one may also take $E_0=L_1(\Om)$, 
where one may check  the analyticity condition \eqref{A1a} with the help of \cite[5.Thm.4]{RunstSickel}. We refrain from giving details and refer to \cite{WalkerSIMA,WalkerJDE} for other examples.
For more concrete applications of global bifurcation results we refer e.g. to~\cite{WalkerCrelle} and to forthcoming research.


\end{document}